\documentclass[10pt]{article}
\usepackage{amsfonts}
\usepackage{amsmath}

\newtheorem{theorem}{Theorem}

\newtheorem{lemma}[theorem]{Lemma}

\newenvironment{proof}{\noindent\\ \noindent\relax{\sc
     Proof}}{{\samepage\par\nopagebreak\hbox
     to\hsize{\hfill$\Box$}}}

\begin{document}

\title{A decomposable branching process in a Markovian environment}
\author{Vladimir Vatutin, Elena Dyakonova, Peter Jagers, Serik Sagitov%
\thanks{%
Corresponding author: serik@chalmers.se} \and \textit{Steklov Mathematical
Institute, Moscow} \\
\textit{Chalmers University of Technology and University of Gothenburg}}
\date{}
\maketitle

\begin{abstract}
A population has two types of individuals, each occupying an island. One of
those, where individuals of type 1 live, offers a variable environment. Type
2 individuals dwell on the other island, in a constant environment. Only
one-way migration ($1\to2$) is possible. We study the asymptotics of the
survival probability in critical and subcritical cases.
\end{abstract}


\section{Introduction}

Multi-type branching process in random environment is a challenging topic
with many motivations from population dynamics (see e.g. \cite{Be,Ha,Sch}). Very little is known in the general case and in this paper we consider a particular two-type branching process with two key restrictions: the process is decomposable and the final type individuals live in a constant environment. 

The subject can be
viewed as a stochastic model for the sizes of a geographically structured
population occupying two islands. Time is assumed discrete, so that one unit
of time represents a generation of individuals, some living on island 1 and
others on island 2. Those on island 1 give birth under influence of a
randomly changing environment. They may migrate to island 2 immediately
after birth, with a probability again depending upon the current
environmental state. Individuals on island 2 do not migrate and their
reproduction law is not influenced by any changing environment. Our main
concern is the survival probability of the whole population.

An alternative interpretation of the model under study might be a population
(type 1) subject to a changing environment, say in the form of a predator
population of stationary but variable size. Its individuals may mutate into
a second type, no longer exposed to the environmental variation (the
predators do not regard the mutants as prey). Our framework may be also suitable for modeling early carcinogenesis,
a process in which mutant clones repeatedly arise and disappear
before one of them becomes established  \cite{SaSe,Tom}.  See \cite{Als} for yet another possible application.

The model framework is that furnished by Bienaym\'{e}-Galton-Watson (BGW)
processes with individuals living one unit of time and replaced by random
numbers of offspring which are conditionally independent given the current
state of the environment. We refer to such individuals as particles in order
to emphasize the simplicity of their lives. Particles of type 1 and 2 are
distinguished according to the island number they are occupying at the
moment of observation. Our main assumptions are:

\begin{itemize}
\item particles of type 1 form a critical or subcritical branching process
in a random environment,

\item particles of type 2 form a critical branching process which is
independent of the environment.
\end{itemize}
Let  $X_{n}$ and $Z_{n}$ be the numbers of particles of type 1 and of type 2, respectively, present at time $n$. Throughout this paper it is assumed (unless otherwise specified) that $X_{0}=1$ and $Z_{0}=0$. We investigate asymptotics of the survival probability $\mathbb{P}\left[X_{n}+Z_{n}>0\right]$ as $n\to\infty$. In all cases addressed here we have 
$$\mathbb{P}\left[ X_{n}>0\right]=o\left(\mathbb{P}[ X_{n}+Z_{n}>0]\right).$$
Therefore, in view of 
\begin{equation}
\mathbb{P}\left[ Z_{n}>0\right] \leq \mathbb{P}\left[ X_{n}+Z_{n}>0\right]
\leq \mathbb{P}\left[ X_{n}>0\right] +\mathbb{P}\left[ Z_{n}>0\right],
\label{EstEqv}
\end{equation}%
we focus on the asymptotic behavior of $\mathbb{P}\left[Z_{n}>0\right]$.

In Section \ref{Sfix} we recall known facts for constant environments. They
will then be compared to the results of this paper on random environments.
In Section \ref{Siid} we describe IID environments (Independent and
Identically Distributed environmental states), and then in Section \ref{Sme}
Markovian environments. The main results of the paper are
\begin{itemize}
\item Theorem \ref{L_general} in Section \ref{ScI} on the
critical case with an IID environment,
\item  Theorem \ref{T_Dsubcrit} in Section \ref{S_S} on the subcritical case with an IID environment,
\item   Theorem \ref{L_Markov} in Section \ref{S_CM} on the critical case with a Markovian environment,
\item Theorem \ref{L_subMark} in Section \ref{S_SM} on subcritical case with a Markovian environment.
\end{itemize}
Theorems  \ref{L_Markov} and \ref{L_subMark} treating the case of Markovian environment are extensions of Theorems  \ref{L_general} and \ref{T_Dsubcrit}  obtained under rather restrictive conditions and yielding qualitatively the same asymptotic behavior as in the case of IID environment.

\begin{equation*}
\mbox{\parbox{10cm}{{\it Notation:} in asymptotic formulae
constants denoted by the same letter $c$ are always assumed to be
fixed and independent of the parameter that tends to infinity (or
zero).}}
\end{equation*}

\section{Two-type decomposable branching processes}

\label{Sfix} Consider a two-type BGW-process initiated at time zero by a
single individual of type 1. We focus on the decomposable case
where type 1 particles may produce particles of types 1 and 2 while the type
2 particles can give birth only to type 2 particles. Put

\begin{itemize}
\item $Y_{n}:=$ the number of type 2 daughters produced by the particles of type 1 present at time $n$, in particular, $Y_0=Z_1$,

\item $T:=$ the first time $n$ when $X_{n}=0$, so that $\{T>n\}=\{X_n>0\}$,

\item $S_{n}:=\sum_{k=0}^{n-1}X_{k}$, so that $S_{\rm T}$ gives the total number
ever  of type 1 particles.

\item $W_{n}:=\sum_{k=0}^{n-1}Y_{k}$, so that $W_{\rm T}$ gives the total number
of type 2 daughters produced by all $S_{\rm T}$ particles of type 1.
\end{itemize}

The aim of this section is to summarize what is already known about such
branching processes in the case of a constant environment. This will pave
our way in terms of notation and basic manipulation with generating
functions towards branching processes in IID random and then Markovian
environments.

If the environment is constant from generation to generation, two-type
decomposable BGW-processes are fully described by a pair of probability
generating functions 
\begin{align*}
f(s_{1},s_{2})& :=\mathbb{E}\left[ s_{1}^{\xi _{1}}s_{2}^{\xi _{2}}\right] ,
\\
h(s)& :=\mathbb{E}\left[ s^{\eta }\right] ,
\end{align*}%
where $\xi _{1}$ and $\xi _{2}$ represent the numbers of daughters of type 1
and 2 of a mother of type 1, while $\eta $ stands for the number of
daughters (necessarily of type 2) of a mother of type 2. Let 
\begin{align*}
\mu _{1}& :=\mathbb{E}\left[ \xi _{1}\right] =\frac{\partial f(s_{1},s_{2})}{%
\partial s_{1}}\left\vert _{s_{1}=s_{2}=1}\right. , \\
\mu _{2}& :=\mathbb{E}\left[ \xi _{1}(\xi _{1}-1)\right] =\frac{\partial
^{2}f(s_{1},s_{2})}{\partial s_{1}^{2}}\left\vert _{s_{1}=s_{2}=1}\right. ,
\\
\theta _{1}& :=\mathbb{E}\left[ \xi _{2}\right] =\frac{\partial
f(s_{1},s_{2})}{\partial s_{2}}\left\vert _{s_{1}=s_{2}=1}\right. , \\
\theta _{2}& :=\mathbb{E}\left[ \xi _{2}(\xi _{2}-1)\right] =\frac{\partial
^{2}f(s_{1},s_{2})}{\partial s_{2}^{2}}\left\vert _{s_{1}=s_{2}=1}\right. ,
\\
m_{1}& :=\mathbb{E}\left[ \eta \right] =h^{\prime }(1), \\
m_{2}& :=\mathbb{E}\left[ \eta (\eta -1)\right] =h^{\prime \prime }(1),
\end{align*}%
be the first two moments of the reproduction laws. Concerning the second
type of particles we assume that 
\begin{equation}
m_{1}=1,\quad m_{2}\in (0,\infty ),  \label{Crit}
\end{equation}%
implying that the probability of extinction 
\begin{equation*}
Q_{n}:=\mathbb{P}[Z_{n}=0|X_{0}=0,Z_{0}=1]
\end{equation*}%
(of a single-type BGW-process evolving in constant environment with the
probability generating function $h(s)$) satisfies \cite[Ch. I.9]
{AN72}
\begin{equation}
1-Q_{n}\sim \frac{2}{m_{2}n},\quad n\rightarrow \infty .  \label{ProbExt}
\end{equation}%
It follows that 
\begin{equation}
a_{n}:=-\log f(1,Q_{n})\sim 1-f(1,Q_{n})\sim {\frac{2\theta _{1}}{m_{2}n}},\
n\rightarrow \infty .  \label{an}
\end{equation}%
We will be interested in two kinds of reproduction regimes for particles of
type 1, critical and subcritical. In the constant environment setting with $%
\mu _{2}\in (0,\infty )$, the critical case corresponds to $\mu _{1}=1$ and
the subcritical case is given by $\mu _{1}\in (0,1)$. In the critical case
with a constant environment we have 
\begin{equation}
\mathbb{P}\left[ X_{n}>0\right] =\mathbb{P}\left[ T>n\right] \sim \frac{2}{%
\mu _{2}n},\ n\rightarrow \infty ,  \label{sucr}
\end{equation}%
and according to \cite[Theorem 1]{FN} 
\begin{equation}
\mathbb{P}\left[ X_{n}+Z_{n}>0\right] \sim \mathbb{P}\left[ Z_{n}>0\right]
\sim \frac{2\sqrt{\theta _{1}}}{\sqrt{m_{2}\mu _{2}n}},\ n\rightarrow \infty
.  \label{Asyl22}
\end{equation}

Next we outline a proof of \eqref{Asyl22} based on the representation 
\begin{align}
\mathbb{P}\left[ Z_{n}>0\right] & =\mathbb{E}\left[ 1-\prod%
\limits_{k=0}^{n-1}Q_{n-k}^{Y_{k}}\right] =\mathbb{E}\left[
1-\prod\limits_{k=0}^{n-1}f^{X_{k}}(1,Q_{n-k})\right]  \notag  \label{repa}
\\
& =\mathbb{E}\left[ 1-e^{-\sum_{k=0}^{n-1}X_{k}a_{n-k}}\right] ,
\end{align}%
preparing for the proof in the random environment case, to be given in
Section \ref{ScI}. Thanks to \eqref{sucr} and  \eqref{EstEqv},
 it is enough to verify that 
\begin{equation*}
\mathbb{P}\left[ Z_{n}>0\right] \sim \frac{2\sqrt{\theta _{1}}}{\sqrt{%
m_{2}\mu _{2}n}},\quad n\rightarrow \infty,
\end{equation*}%
in order to prove \eqref{Asyl22}. However, by the branching property the total progeny of a single-type
branching process $S_{\rm T}$ is 1 plus $\xi _{1}$ independent daughter copies
of $S_{\rm T}$. In terms of the Laplace transform 
\begin{equation*}
\phi (\lambda )=e^{-\lambda }f(\phi (\lambda ),1),
\end{equation*}%
where $\phi (\lambda ):=\mathbb{E}[e^{-\lambda S_{\rm T}}]$. As $\lambda
\rightarrow 0$, a Taylor expansion of $f(\phi (\lambda ),1)$
as a function of $1-\phi (\lambda )$ yields 
\begin{equation}
1-\phi (\lambda )=1-e^{-\lambda }+e^{-\lambda }\mu _{1}(1-\phi (\lambda
))-e^{-\lambda }{\frac{\mu _{2}}{2}}(1-\phi (\lambda ))^{2}(1+o(1)).
\label{phi}
\end{equation}%
For $\mu _{1}=1$, after removing the negligible terms, we get a quadratic equation whose solution shows that 
\begin{equation*}
1-\phi (\lambda )\sim \sqrt{{2\lambda /\mu _{2}}},\quad \lambda \rightarrow
0.
\end{equation*}%
Replacing $\lambda$ by $a_n$ and using \eqref{an}, we obtain
\begin{equation*}
\mathbb{E}\left[ 1-e^{-S_{\rm T}a_{n}}\right] \sim \frac{2\sqrt{\theta _{1}}}{%
\sqrt{m_{2}\mu _{2}n}},\ n\rightarrow \infty.
\end{equation*}%
It remains to verify, see \eqref{repa}, that 
\begin{equation*}
\mathbb{E}\left[ 1-e^{-\sum_{k=0}^{n-1}X_{k}a_{n-k}}\right] \sim \mathbb{E}%
\left[ 1-e^{-S_{\rm T}a_{n}}\right] ,\ n\rightarrow \infty .
\end{equation*}%
This holds, indeed, since by \eqref{sucr} and for any fixed $\epsilon >0$
the probability $\mathbb{P}[T>n\epsilon ]$ is much smaller than the target
value of order $c/\sqrt{n}$. (In \cite{VS} and \cite{Z} infinite second
moments in decomposable two-type critical processes were allowed.)

On the other hand, in the subcritical case \eqref{phi} implies that 
\begin{equation*}
1-\phi(\lambda)\sim\lambda/(1-\mu_1),\quad \lambda\to0,
\end{equation*}
so that by \eqref{an} 
\begin{equation*}
\mathbb{E}\left[ 1-e^ {- S_{\rm T}a_{n}} \right] \sim \frac{2\theta_1}{%
m_2(1-\mu_1)n}, \ n\rightarrow \infty.
\end{equation*}
In view of $\mathbb{P}\left[X_{n}>0\right]\sim c\mu_1^n$ we conclude that in
the subcritical case 
\begin{equation}  \label{subcr}
\mathbb{P}\left[ X_{n}+Z_{n}>0\right] \sim\mathbb{P}\left[Z_{n}>0\right]\sim%
\frac{2\theta_1}{m_2(1-\mu_1)n}, \ n\to\infty.
\end{equation}%
See \cite{Og} for a comprehensive study of subcritical decomposable
branching processes in a constant environment.

\section{Branching processes in a random environment}

\label{Sre}

A randomly changing environment for BGW-processes is modeled by a random
sequence of probability generating functions for the offspring distributions
of consecutive generations. Throughout this paper we assume that the
offspring distribution for type 2 particles is the same across the different
states of the environment and characterized by the same generating function $%
h(s)$. This restriction greatly simplifies analysis still allowing new
interesting asymptotic regimes.

We consider two types of stationarily changing environments: IID and
Markovian.

\subsection{IID environment}
\label{Siid} 
Our description of the IID environment case starts with a
simple illustration based on just two alternative bivariate generating
functions $f^{(1)}(s_1,s_2)$ and $f^{(2)}(s_1,s_2)$ with mean offspring
numbers $(\mu^{(1)}_1,\theta^{(1)}_1)$ and $(\mu^{(2)}_1,\theta^{(2)}_1)$
respectively. 
We assume that at each time $n$ the environment is say "good" with
probability $\pi_1$, so that the type 1 particles reproduce independently
according $f^{(1)}(s_1,s_2)$, and with probability $\pi_2=1-\pi_1$ the
environment is "bad" and particles of type 1 reproduce according to the $%
f^{(2)}(s_1,s_2)$ law. In other words, the generating function $f(s_1,s_2)$
should be treated as a random function having distribution
\[\mathbb P[f(s_1,s_2)=f^{(1)}(s_1,s_2)]= \pi_1,\quad \mathbb P[f(s_1,s_2)=f^{(2)}(s_1,s_2)]= \pi_2.\]
In particular,
the vector of the mean offspring numbers $(\mu_1,\theta_1)$ takes values $%
(\mu^{(1)}_1,\theta^{(1)}_1)$ and $(\mu^{(2)}_1,\theta^{(2)}_1)$ with
probabilities $\pi_1$ and $\pi_2$.

More generally, our two-type branching process in an IID random environment
is characterized (besides the fixed reproduction law $h(s)$ for the type 2
particles) by a sequence of generating functions $\{f_{n}(s_{1},s_{2})%
\}_{n=0}^{\infty }$ independently drawn from a certain distribution over
probability generating functions so that 
\begin{equation}
f_{n}(s_{1},s_{2})\overset{d}{=}f(s_{1},s_{2}).  \label{idd}
\end{equation}%
In this setting the respective conditional moments $\mu _{1}$, $\mu _{2}$, $%
\theta _{1}$, and $\theta _{2}$ should be treated as random variables. An
important role is played by the random variable $\zeta :=\log \mu _{1}$
representing the step size of the so-called associated random walk \cite%
{AGKV} formed by the partial sums $\zeta _{0}+\dots +\zeta _{n-1}$ with $%
\zeta _{i}\overset{d}{=}\zeta $. 
\begin{equation}
\mbox{\parbox{10cm}{{\it Notation:} characteristics of the
reproduction law in generation $n$ are denoted by adding an extra lower
index $n$ to the generic notation, like in \eqref{idd}. }}  \label{abrk}
\end{equation}

\subsection{Markovian environment}

\label{Sme} One way to relax the IID assumption on the environment is to
allow for Markovian dependence among its consecutive states. We implement
this by modelling changes in terms of an irreducible aperiodic
positive recurrent Markov chain $\{e_{n}\}_{n=0}^{\infty }$ with countably
many states $\{1,2,\ldots \}$. 
Assuming a stationary initial distribution $(\pi _1,\pi_2,\ldots)$, we associate with each
state $i$ of this chain a probability generating function $%
f^{(i)}(s_{1},s_{2})$, so that the changing environment for the branching
process is governed by the sequence of identically reproduction laws 
\begin{equation*}
f_{n}(s_{1},s_{2}):=f^{(e_{n})}(s_{1},s_{2}),\quad n=0,1,\dots
\end{equation*}%
with Markovian dependence. Due to the stationarity we can again write %
\eqref{idd} and use the same notation for the marginal moments of the
reproduction laws as in the IID case.

To build a bridge to the IID environment case we use an embedding through a
sequence of regeneration moments $\{\tau _{k}\}_{k=0}^{\infty }$ defined as 
\begin{equation}
\tau _{0}:=0,\tau _{k+1}:=\min \{n>\tau _{k}:e_{n}=e_{0}\}.  \label{reg}
\end{equation}%
The times $\tau _{k+1}-\tau _{k}$ between consecutive regenerations are
independent and all distributed as $\tau :=\tau _{1}$. The embedded process $(\hat X_{n},\hat Z_{n})$ defined as 
\begin{equation*}
(\hat X_{n},\hat Z_{n}):=(X_{\tau _{n}},Z_{\tau _{n}}),\quad
n=0,1,\dots
\end{equation*}%
is a decomposable branching process in an IID environment with two types of
particles $\hat 1$ and $\hat 2$ and conditional reproduction
generating functions 
\begin{align}
\hat f(s_{1},s_{2})& :=f^{(e_{0})}\big(f^{(e_{1})}\big(\dots \big(%
f^{(e_{\tau -1})}(s_{1},s_{2}),h(s_{2})\big)\dots \big),h_{\tau -1}(s_{2})%
\big),  \label{DefF} \\
\hat h(s)& :=h(h(\dots h(s)\dots ))=h_{\tau }(s),  \label{Defh}
\end{align}%
where $h_{k}(s)$ stands for the $k$-fold iteration of $h(s)$.
\begin{quote}
 \mbox{\parbox{10cm}{{\it Notation:} for all characteristics of the
embedded process $(\hat X_{n},\hat Z_{n})$ and related constants appearing in the asymptotic formulae we use the same notation as for the process $(X_{n},Z_{n})$ in the IID case just adding the hat sign. }}  

\end{quote}

The key difference from the IID case is that the reproduction law for the $%
\hat 2$-type particles is dependent on the random environment.
However, this dependence is of specific nature which we are able to manage
using the law of large numbers for renewal processes. Notice that on its
own the $\hat 2$-type particles form a so-called degenerate critical
branching process in an IID random environment \cite{AGKV}: its conditional
offspring mean is deterministic $\hat m_{1}=1$. Meanwhile, the
conditional variance is random $\hat m_{2}=\tau m_{2}$.

Taking the first and second order derivatives of \eqref{DefF}, we can
express the moments of the reproduction law of the embedded process in terms of the moments of
the consecutive reproduction laws with Markovian dependence. In what follows
we use \eqref{abrk} again, while keeping in mind that the sequence $(\mu
_{1,k},\mu _{2,k},\theta _{1,k},\theta _{2,k})_{k=0}^{\tau -1}$ now consists
of dependent random vectors. It can be shown that 
\begin{align*}
\hat \mu _{1}& =\prod_{k=0}^{\tau -1}\mu _{1,k},\quad \hat \mu _{2}=\hat \mu _{1}\sum\limits_{k=0}^{\tau -1}\frac{\mu _{2,k}}{\mu _{1,k}}%
\prod_{i=k+1}^{\tau -1}\mu _{1,i}, \\
\hat \theta _{1}& =\sum\limits_{k=0}^{\tau -1}\theta
_{1,k}\prod_{i=0}^{k-1}\mu _{1,i},
\end{align*}%
where, as usual, the product of the elements of an empty set is one. Furthermore, setting 
\begin{equation*}
A_{k,n}=\sum_{j=k}^{n}\theta _{1,j}\prod\limits_{i=k}^{j-1}\mu _{1,i}
\end{equation*}%
we can write 
\begin{align*}
\hat \theta _{2}& =\sum_{k=0}^{\tau -1}\theta
_{2,k}\prod\limits_{i=0}^{k-1}\mu _{1,i} \\
& +\sum_{k=0}^{\tau -2}\left\{ \mu _{2,k}A_{k+1,\tau -1}^{2}+2\mu
_{1,k}\theta _{1,k}A_{k+1,\tau -1}+\sigma ^{2}\left( \tau -1-k\right) \theta
_{1,k}\right\} \prod\limits_{i=0}^{k-1}\mu _{1,i}.
\end{align*}

\begin{lemma}\hspace{-1.5mm}. \label{pro}
Let
\begin{equation}
\sum_{k=0}^{\infty }\mathbb{E}\left[ |\zeta _{k}|1_{\{\tau \geq k+1\}}\right]
<\infty.  \label{fub}
\end{equation}%
For the following sum of a random number of random variables 
\begin{equation*}
\hat \zeta :=\sum_{k=0}^{\tau -1}\zeta _{k},\quad \zeta _{k}:=\log \mu
_{1,k}=\zeta (e_{k}),
\end{equation*}%
a version of the Wald identity holds: $\mathbb{E}[\hat  \zeta]=\mathbb{E}[ \tau ]
\mathbb{E}[\zeta]$.
\end{lemma}
\begin{proof}.
For any state $j$ consider the function
\begin{equation*}
\mu _{j}(i):=\sum_{n=0}^{\infty }\mathbb{P}[e_{n}=i,\tau >n|e_{0}=j].
\end{equation*}%
According to  \cite[Th.6.5.2]{Du} this defines a stationary measure which is necessarily of the form $\mu_{j}(i)=c_{j}\pi _{i}$. The constants $c_j$ are such that
\begin{align*}
\sum_{j=1}^{\infty }\pi _{j}c_{j}& =\sum_{j=1}^{\infty }\pi
_{j}\sum_{i=1}^{\infty }\mu _{j}(i)=\sum_{j=1}^{\infty }\pi
_{j}\sum_{n=0}^{\infty }\sum_{i=1}^{\infty }\mathbb{P}[e_{n}=i,\tau
>n|e_{0}=j] \\
& =\sum_{j=1}^{\infty }\pi _{j}\mathbb{E}\left[\tau |e_{0}=j\right] =%
\mathbb{E}[\tau].
\end{align*} 
It follows  that
\begin{align*}
\mathbb{E}[\hat \zeta]& =\sum_{k=1}^{\infty
}\sum_{n=0}^{k-1}\mathbb{E}\left[ \zeta _{n}1_{\{\tau =k\}}\right]
=\sum_{n=0}^{\infty }\mathbb{E}\left[ \zeta (e_{n})1_{\{\tau >n\}}\right] \\
& =\sum_{n=0}^{\infty }\sum_{i=1}^{\infty }\zeta (i)\mathbb{P}[e_{n}=i,\tau>n]=\sum_{i=1}^{\infty }\zeta (i)\sum_{j=1}^{\infty }\pi _{j}\mu_{j}(i)\\
&=\sum_{i=1}^{\infty }\zeta (i)\pi _{i}\mathbb{E}[\tau]=%
\mathbb{E}[\tau] \mathbb{E}[\zeta].
\end{align*}%
\end{proof}

%

Developing the example of two environmental states from Section \ref{Siid},
let us consider a Markov chain $\{e_{n}\}_{n=0}^{\infty }$ with transition
probabilities 
\begin{equation*}
\left( 
\begin{array}{cc}
1-d\pi _{2} & d\pi _{2} \\ 
d\pi _{1} & 1-d\pi _{1}%
\end{array}%
\right) ,\quad 0<d<\min \Big({\frac{1}{\pi _{1}}},{\frac{1}{\pi _{2}}}\Big)
\end{equation*}%
and a stationary distribution $(\pi _{1},\pi _{2})$. (Notice that $d=1$
corresponds to the IID case.) Under stationarity the
regeneration time satisfies
\begin{align*}
\mathbb{P}\left[ \tau =1\right] & =\pi _{1}(1-d\pi _{2})+\pi _{2}(1-d\pi
_{1})=1-2\pi _{1}\pi _{2}d, \\
\mathbb{P}\left[ \tau =k\right] & =\pi _{1}d\pi _{2}(1-d\pi _{1})^{k-2}d\pi
_{1}+\pi _{2}d\pi _{1}(1-d\pi _{2})^{k-2}d\pi _{2} \\
& =d\pi _{1}\pi _{2}\Big(d\pi _{1}(1-d\pi _{1})^{k-2}+d\pi _{2}(1-d\pi
_{2})^{k-2}\Big),\quad k\geq 2,
\end{align*}%
implying that
\begin{equation*}
\mathbb{E}\left[ \tau -1\right] =1,\quad \mathbb{E}\left[ \tau (\tau -1)%
\right] ={\frac{2}{d\pi _{1}\pi _{2}}}-{\frac{4}{d}}.
\end{equation*}%
If $(b_{1},b_{2})$ are the two possible values for $\zeta $, we can write $%
\mathbb{E}[\zeta ]=\pi _{1}b_{1}+\pi _{2}b_{2}$ and 
\begin{align*}
\mathbb{E}[\hat \zeta]& =\mathbb{E}\left[ \mathbb{E}[\hat \zeta|\tau ] ;\tau =1\right] +\mathbb{E}\left[ \mathbb{E}%
[\hat  \zeta|\tau ];\tau \geq 2\right] \\
& =\pi _{1}(1-d\pi _{2})b_{1}+\pi _{2}(1-d\pi _{1})b_{2} \\
& \quad +\pi _{1}d\pi _{2}\left( b_{1}+{\frac{b_{2}}{d\pi _{1}}}\right) +\pi
_{2}d\pi _{1}\left( b_{2}+{\frac{b_{1}}{d\pi _{2}}}\right) =2\mathbb{E}[\zeta],
\end{align*}%
in full agreement with Lemma \ref{pro}. %

\section{Critical processes in IID environment}

\label{ScI}

The single type critical branching process with an IID environment displays
an asymptotic behavior that is in stark contrast with the constant
environment formula \eqref{sucr}. According to \cite[Th.1]{GK00}, if

\begin{equation}
\mathbb{E}[\zeta ]=0,\quad \mathrm{Var}[\zeta] \in
\left( 0,\infty \right),  \label{ReCritical}
\end{equation}
\begin{equation}
\mathbb{E}\left[\mu _{2}\mu _{1}^{-2}\left( 1+\max \left( 0,\log \mu
_{1}\right) \right) \right] <\infty ,  \label{ReC}
\end{equation}%
then for some positive constant $c$
\begin{equation}
\mathbb{P}\left[ X_{n}>0\right] =\mathbb{P}\left[ T>n\right] \sim \frac{c}{%
\sqrt{n}},\ n\rightarrow \infty .  \label{koz}
\end{equation}%
(A much more general limit theorem is obtained in \cite{AGKV}.) The
following theorem shows that in the decomposable case the difference between
 constant and random environments 
is even more striking. For constant environments the survival probability
decays as $c/\sqrt{n}$, see \eqref{Asyl22}, but in random environments the
decay is like $c/\log n$.

\begin{theorem}\hspace{-1.5mm}.
\label{L_general} Consider a critical decomposable branching process in an
IID environment satisfying \eqref{Crit} \eqref{ReCritical}, \eqref{ReC}, and 
\begin{equation}
\mathbb{E}\left[ \mu _{1}^{-1}\right] <\infty .  \label{Aff}
\end{equation}%
If for some positive $\alpha $ 
\begin{align}
\mathbb{P}\left[ \theta _{1}<1/x\right]& =o\left( (\log x)^{-3-\alpha
}\right) ,\ x\rightarrow \infty ,  \label{c0} \\
\mathbb{P}\left[ \theta _{1}>x\right]& =o\left( (\log x)^{-3-\alpha
}\right) ,\ x\rightarrow \infty ,  \label{As_theta} \\
\mathbb{P}\left[\theta _{2}>x\theta _{1}\right]& =o\left( (\log
x)^{-3-\alpha }\right) ,\ x\rightarrow \infty ,  \label{lambda}
\end{align}%
then there exists a constant $K_{0}$ such that 
\begin{equation*}
\mathbb{P}\left[ Z_{n}>0\right]
\sim \frac{K_{0}}{\log n},\ n\rightarrow \infty .  
\end{equation*}
\end{theorem}

Before turning to the proof, we make some comments on the conditions and statement of
this theorem. 
\begin{equation}
\mbox{\parbox{10cm}{{\it Notation:}  we will often use the
abbreviations $x_a:=(\log x)^{2+a}$ and $n_a:=(\log n)^{2+a}$.}}  \label{abr}
\end{equation}%
Conditions (\ref{c0}), (\ref{As_theta}), and (\ref{lambda}) are needed for
the following properties to hold for any fixed $\varepsilon >0$, recall notation agreements \eqref{abrk} and \eqref{abr},
\begin{align}
\mathbb{P}\left[ \min_{0\leq k\leq x_{\alpha }}\theta _{1,k}<x^{-\varepsilon
}\right] & =o\left( \frac{1}{\log x}\right) ,\ x\rightarrow \infty ,
\label{oo} \\
\mathbb{P}\left[ \max_{0\leq k\leq x_{\alpha }}\theta _{1,k}>x^{\varepsilon
}\right] & =o\left( \frac{1}{\log x}\right) ,\ x\rightarrow \infty ,
\label{ooo} \\
\mathbb{P}\left[ \max_{0\leq k\leq x_{\alpha }}(\theta _{2,k}/\theta
_{1,k})>x^{\varepsilon }\right] & =o\left( \frac{1}{\log x}\right) ,\
x\rightarrow \infty .  \label{ooO}
\end{align}%
Each of them is proven via an intermediate step like 
\begin{equation*}
\mathbb{P}\left[ \min_{0\leq k\leq x_{\alpha }}\theta _{1,k}<x^{-\varepsilon
}\right] \leq x_{\alpha }\mathbb{P}\left( \theta _{1}<x^{-\varepsilon
}\right) 
\end{equation*}%
relying on the IID assumption for consecutive environmental states. The
constant $K_{0}$ in the statement of Theorem \ref{L_general}  is the same as in the
asymptotic formula from \cite{Af99} concerning the total number $S_{\rm T}$ of
particles of type 1 ever appeared in the process: 
\begin{equation}
\mathbb{P}\left[ S_{\rm T}>x\right] \sim \frac{K_{0}}{\log x},\ x\rightarrow
\infty.  \label{AsTotal1}
\end{equation}%
This constant has a complicated nature and is not
further explained here. It is necessary to mention that the representation (%
\ref{AsTotal1}) has been proved in \cite{Af99} under conditions %
\eqref{ReCritical}, \eqref{ReC}, and \eqref{Aff} only for the case when the
probability generating functions $f_{n}(s,1)$ are linear-fractional with
probability 1. However, the latter restriction is easily removed using the
results established later on for the general case in \cite{GK00} and \cite%
{AGKV}.

Our proof of Theorem \ref{L_general} uses the next lemma.
\begin{lemma}
\label{L_estJ} Consider conditional moments of the entity $W_{n}$ defined
at the beginning of Section \ref{Sfix}: 
\begin{equation*}
S_{n}^{(i)}:=\sum_{k=0}^{n-1}X_{k}\theta _{i,k},\ i=1,2.
\end{equation*}%
Under conditions \eqref{ReCritical}, \eqref{c0}, \eqref{As_theta}, and %
\eqref{lambda}, 
\begin{equation*}
\mathbb{P}\left[ S_{\rm T}^{(1)}>x\right] \sim \frac{K_{0}}{\log x},\quad
x\rightarrow \infty.
\end{equation*}%
For any fixed $\epsilon >0$, in the notation from \eqref{abr}, 
\begin{equation*}
\mathbb{P}\left[ S_{\rm T}^{(2)}>n^{\varepsilon }S_{\rm T}^{(1)};T\leq n_{\alpha
}\right] =o\left( \frac{1}{\log n}\right) ,\ n\rightarrow \infty .
\end{equation*}
\end{lemma}

\begin{proof}.
 For any fixed $\epsilon >0$, 
\begin{eqnarray*}
\mathbb{P}\left[ S_{\rm T}^{(1)}>x\right] &\geq &\mathbb{P}\left[
S_{\rm T}^{(1)}>x;T\leq x_{\alpha };\min_{0\leq k\leq T}\theta
_{1,k}>x^{-\varepsilon }\right] \\
&\geq &\mathbb{P}\left[ S_{\rm T}>x^{1+\varepsilon }\right] -\mathbb{P}\left[
T>x_{\alpha }\right] -\mathbb{P}\left[ \min_{0\leq k\leq x_{\alpha }}\theta
_{1,k}\leq x^{-\varepsilon }\right] .
\end{eqnarray*}%
Notice that according to \eqref{koz} 
\begin{equation}
\mathbb{P}\left[ T>(\log x)^{2+\varepsilon }\right] =o\left( \frac{1}{\log x}%
\right) ,\ x\rightarrow \infty ,\mbox{ for any fixed $\varepsilon>0$}.
\label{Tx}
\end{equation}%
Thus, using (\ref{oo}) and (\ref{AsTotal1}) we get 
\begin{equation*}
\liminf_{x\rightarrow \infty }\left\{ \log x\cdot \mathbb{P}\left[
S_{\rm T}^{(1)}>x\right] \right\} \geq \liminf_{x\rightarrow \infty }\left\{
\log x\cdot \mathbb{P}\left[ S_{\rm T}>x^{1+\varepsilon }\right] \right\} \geq
K_{0}/(1+\epsilon ).
\end{equation*}%
To obtain a similar estimate from above we write, recalling \eqref{abr},%
\begin{eqnarray*}
\mathbb{P}\left[ S_{\rm T}^{(1)}>x\right] &\leq &\mathbb{P}\left[
S_{\rm T}^{(1)}>x;T\leq x_{\alpha };\max_{0\leq k\leq T}\theta _{1,k}\leq
x^{\varepsilon }\right] \\
&+&\mathbb{P}\left[ T>x_{\alpha }\right] +\mathbb{P}\left[ T\leq x_{\alpha
};\max_{0\leq k\leq T}\theta _{1,k}>x^{\varepsilon }\right] \\
&\leq &\mathbb{P}\left[ S_{\rm T}>x^{1-\varepsilon }\right] +\mathbb{P}\left[
T>x_{\alpha }\right] +\mathbb{P}\left[ \max_{0\leq k\leq x_{\alpha }}\theta
_{1,k}>x^{\varepsilon }\right],
\end{eqnarray*}%
which together with \eqref{ooo}, (\ref{AsTotal1}), and \eqref{Tx} yields 
\begin{eqnarray*}
\limsup_{x\rightarrow \infty }\left\{ \log x\cdot \mathbb{P}\left[
S_{\rm T}^{(1)}>x\right] \right\} &\leq &\limsup_{x\rightarrow \infty }\left\{
\log x\cdot \mathbb{P}\left[ S_{\rm T}>x^{1-\varepsilon }\right] \right\} \\
&\leq &K_{0}/(1-\epsilon ).
\end{eqnarray*}%
Finally, according to \eqref{ooO} 
\begin{eqnarray*}
\mathbb{P}\left[ S_{\rm T}^{(2)}>n^{\varepsilon }S_{\rm T}^{(1)};T\leq n_{\alpha
}\right] &\leq &\mathbb{P}\left[ \max_{1\leq k\leq T}(\theta _{2,k}/\theta
_{1,k})>n^{\varepsilon };T\leq n_{\alpha }\right] \\
&=&o\left( \frac{1}{\log n}\right) ,\ n\rightarrow \infty .
\end{eqnarray*}
 
\end{proof}
\begin{proof} {\sc  of Theorem \ref{L_general}.} We will show that 
\begin{align}
\limsup_{n\rightarrow \infty }\left\{ \log n\cdot \mathbb{P}%
[Z_{n}>0]\right\} \leq K_{0} \leq \liminf_{n\rightarrow \infty }\left\{ \log n\cdot \mathbb{P}%
[Z_{n}>0]\right\}  \label{Up}
\end{align}%
using a counterpart of %
\eqref{repa} 
\begin{align}
\mathbb{P}\left[ Z_{n}>0\right] & =\mathbb{E}\left[ 1-\prod%
\limits_{k=0}^{n-1}Q_{n-k}^{Y_{k}}\right] =\mathbb{E}\left[
1-\prod\limits_{k=0}^{n-1}f_{k}^{X_{k}}(1,Q_{n-k})\right]  \notag
\label{repZ} \\
& =\mathbb{E}\left[ 1-\exp \left\{ \sum_{k=0}^{n-1}X_{k}\log
f_{k}(1,Q_{n-k})\right\} \right] ,
\end{align}%
and Lemma \ref{L_estJ}.

First we prove the second inequality in  \eqref{Up}. It follows from \eqref{repZ} and the
monotonicity of $Q_{n}$ that for any fixed $\epsilon \in (0,1)$ 
\begin{equation*}
\mathbb{P}\left[ Z_{n}>0\right] \geq \mathbb{E}\left[ 1-\exp \left\{
\sum_{k=0}^{T-1}X_{k}\log f_{k}(1,Q_{n})\right\} ;S_{\rm T}^{(2)}\leq
n^{\varepsilon }S_{\rm T}^{(1)},T\leq n_{\alpha }\right] .
\end{equation*}%
Recall that $\log (1-x)\leq -x$ and 
\begin{equation*}
f(1,s)\leq 1+\theta _{1}\left( s-1\right) +{\frac{\theta _{2}}{2}}\left(
1-s\right) ^{2},
\end{equation*}%
with the latter inequality being valid thanks  to the monotonicity of the second derivative of the generating
function. Therefore, 
\begin{equation*}
\log f(1,s)\leq -\theta _{1}\left( 1-s\right) +(\theta _{2}/2)\left(
1-s\right) ^{2}
\end{equation*}%
and 
\begin{eqnarray*}
\sum_{k=0}^{T-1}X_{k}\log f_{k}(1,Q_{n}) &\leq &-\left( 1-Q_{n}\right)
\sum_{k=0}^{T-1}X_{k}\theta _{1,k}+\frac{\left( 1-Q_{n}\right) ^{2}}{2}%
\sum_{k=0}^{T-1}X_{k}\theta _{2,k} \\
&\leq &-c_{1}n^{-1}\sum_{k=0}^{T-1}X_{k}\theta
_{1,k}+c_{2}n^{-2}\sum_{k=0}^{T-1}X_{k}\theta _{2,k},
\end{eqnarray*}%
where the last inequality is due to \eqref{ProbExt}.
It follows that given $S_{\rm T}^{(2)}\leq n^{\varepsilon }S_{\rm T}^{(1)}$,
\begin{equation*}
\sum_{k=0}^{T-1}X_{k}\log f_{k}(1,Q_{n})\leq -cn^{-1}S_{\rm T}^{(1)}
\end{equation*}%
for sufficiently large $n$. As a result, we see that for large $n$%
\begin{align*}
\mathbb{P}\left[ Z_{n}>0\right] & \geq \mathbb{E}\left[
1-e^{-cn^{-1}S_{\rm T}^{(1)}};S_{\rm T}^{(2)}\leq n^{\varepsilon }S_{\rm T}^{(1)},T\leq
n_{\alpha }\right] \\
& \geq \mathbb{E}\left[ 1-e^{-cn^{-1}S_{\rm T}^{(1)}}\right] -\mathbb{P}\left[
T>n_{\alpha }\right] -\mathbb{P}\left[ S_{\rm T}^{(2)}>n^{\varepsilon
}S_{\rm T}^{(1)};T\leq n_{\alpha }\right].
\end{align*}
Now, to finish the proof of the second inequality in  \eqref{Up} it remains to use \eqref{Tx}, Lemma \ref{L_estJ}, and 
\begin{equation*}
\mathbb{E}\left[ 1-e^{-\lambda S_{\rm T}^{(1)}}\right] \sim \frac{K_{0}}{\log
(1/\lambda )},\ \lambda \rightarrow 0,
\end{equation*}%
which again due to Lemma \ref{L_estJ} follows from the Tauberian theorem \cite[Ch. XIII.5, Th.4]{Fe2} applied to the right hand side of 
\begin{equation*}
\lambda ^{-1}\mathbb{E}\left[ 1-e^{-\lambda S_{\rm T}^{(1)}}\right]
=\int_{0}^{\infty }\mathbb{P}\big[ S_{\rm T}^{(1)}>x\big] e^{-\lambda x}dx.
\end{equation*}

Next, we verify the first inequality in \eqref{Up}.
From the estimates $\log
(1-x)\geq -2x$, valid for $x\in \left( 0,1/2\right)$, and
$
f(1,s)\geq 1+\theta _{1}\left( s-1\right)$, we conclude
that for all sufficiently large $n$
\begin{align*}
 \mathbb{E}&\left[ 1-\exp \left\{ \sum_{k=0}^{T-1}X_{k}\log
f_{k}(1,Q_{n-T})\right\} ;T\leq n_{\alpha };\max_{0\leq k\leq T}\theta
_{1,k}\leq n^{\varepsilon }\right] \\
&\leq \mathbb{E}\left[ 1-\exp \left\{ \sum_{k=0}^{T-1}X_{k}\log \Big(1-{c\theta_{1,k}\over n}\Big)\right\} ;T\leq n_{\alpha };\max_{0\leq k\leq T}\theta
_{1,k}\leq n^{\varepsilon }\right] \\
&\leq \mathbb{E}\left[ 1-\exp \left\{ -2cn^{-1}\sum_{k=0}^{T-1}X_{k}\theta
_{1,k}\right\} ;T\leq n_{\alpha };\max_{0\leq k\leq T}\theta _{1,k}\leq
n^{\varepsilon }\right] \\
&\leq \mathbb{E}\left[ 1-e^{-2cn^{\varepsilon -1}S_{\rm T}^{(1)}}\right].
\end{align*}%
Thus,
\begin{align*}
\mathbb{P}\left[ Z_{n}>0\right]&\leq \mathbb{E}\left[ 1-\exp \left\{
\sum_{k=0}^{T-1}X_{k}\log f_{k}(1,Q_{n-T})\right\} \right]\\
& \leq \mathbb{E}\left[ 1-e^{-2cn^{\varepsilon
-1}S_{\rm T}^{(1)}}\right] +\mathbb{P}\left[ T>n_{\alpha }\right] +\mathbb{P}%
\left[ \max_{0\leq k\leq n_{\alpha }}\theta _{1,k}>n^{\varepsilon }\right],
\end{align*}%
and \eqref{Up} follows due to \eqref{ooo} and \eqref{Tx}.
 
\end{proof}

\section{The subcritical case with an IID environment\label{S_S}}

We continue studying BGW-processes in IID environment but now assume 
\begin{equation}
\mathbb{E}\left[ \zeta \right]<0,\quad \mathrm{Var}\left[ \zeta \right] \in
\left( 0,\infty \right)  \label{ReS}
\end{equation}%
instead of \eqref{ReCritical}. Results rely upon a theorem from \cite{V10}
giving the asymptotics for $\mathbb{P}[W_{\rm T}>x]$ as $x\rightarrow \infty $.
It requires an important technical assumption viz. the existence of a constant $%
\kappa $ such that 
\begin{equation}
\mathbb{E}\left[ e^{\kappa \zeta }\right] =\mathbb{E}\left[ \mu _{1}^{\kappa
}\right] =1,\quad 0<\kappa <\infty .  \label{defkap}
\end{equation}%
%
%
%
%
%
%
%
%
%
%
%
%
%
%
If, in addition, for some $\delta >0$ 
\begin{equation}
0<\mathbb{E}\left[\xi _{2}^{\kappa +\delta }\right] <\infty ,~~\mathbb{E}%
\left[ \theta _{1}^{\kappa }\right] <\infty ,  \label{cond}
\end{equation}%
and either
\begin{equation}
\kappa >1,\ \mathbb{E}\left[ \mathbf{|}\xi _{1}|^{\kappa }\right] <\infty ,
\label{condg0}
\end{equation}%
or
\begin{equation}
0<\kappa \leq 1,\ \mathbb{E}\left[ |\mu _{2}-\mu _{1}^{2}|^{\kappa }+|\theta
_{2}-\theta _{1}^{2}|^{\kappa }\right] <\infty ,  \label{condg}
\end{equation}%
then, according to \cite{V10}, there exists a constant $C_{\kappa }\in
\left( 0,\infty \right) $ such that

\begin{equation}
\mathbb{P}\left[ W_{\rm T}>x\right] \sim C_{\kappa }x^{-\kappa } ,\;x\rightarrow
\infty .  \label{AsTotal1su}
\end{equation}
It is also known \cite{AGKV2,ABKV1,ABKV2} that under (\ref{ReS}) and (\ref{defkap}) 
\begin{equation}
\mathbb{P}\left[ X_{n}>0\right] =\mathbb{P}\left[ T>n\right] =o\left(
A^{n}\right) \mbox { for some constant
$A\in \left( 0,1\right) $}.  \label{3}
\end{equation}%

\begin{theorem}\hspace{-1.5mm}.
\label{T_Dsubcrit} If conditions \eqref{ReS}, \eqref{defkap}, \eqref{cond}
and either \eqref{condg0} or \eqref{condg} hold, then 
\begin{equation}
\mathbb{P}\left[ Z_{n}>0\right]
\sim K_{\kappa }\cdot q_{\kappa }(n),\quad n\rightarrow \infty ,
\label{Asym222}
\end{equation}%
for some positive constant $K_{\kappa }$, given by \eqref{ReprK} below,
where 
\begin{equation}
q_{\kappa }(n)=\left\{ 
\begin{array}{ll}
n^{-\kappa }, & \mbox{if }\kappa <1, \\ 
n^{-1}\log n, & \mbox{if }\kappa =1, \\ 
n^{-1}, & \mbox{if }\kappa >1.%
\end{array}%
\right.  \label{qka}
\end{equation}
\end{theorem}
\begin{proof}. Referring to (\ref{3}), put $B_n:={2\log n\over \log(A^{-1})}$ and notice that 
\[\mathbb{P}\left[
T>B_n\right]=o(n^{-2}),\quad n\to\infty.\]
From the first equality in \eqref{repZ} and
the evident inequality $Q_{n-k}\leq Q_{n}$ we obtain for $n\geq B_n$ 
\begin{eqnarray*}
\mathbb{P}\left[ Z_{n}>0\right] &\geq &\mathbb{E}\left[ 1-\prod%
\limits_{k=0}^{n-1}Q_{n}^{Y_{k}};T\leq B_n\right] \\
&=&\mathbb{E}\left[ 1-e^{W_{\rm T}\log Q_{n}};T\leq B_n\right] \geq \mathbb{E}\left[ 1-e^{W_{\rm T}\log Q_{n}}\right] -\mathbb{P}\left[
T>B_n\right].
\end{eqnarray*}%
On the other hand, we
have a similar upper bound 
\begin{eqnarray*}
\mathbb{P}\left[ Z_{n}>0\right] &\leq &\mathbb{E}\left[ 1-Q_{n-T}^{W_{\rm T}};T%
\leq B_n\right] +\mathbb{P}\left[T>B_n\right] \\
&\leq &\mathbb{E}\left[ 1-e^{W_{\rm T}\log Q_{n-N\log n}}\right] +\mathbb{P}%
\left[ T>B_n\right] .
\end{eqnarray*}

It remains to observe that due to \eqref{ProbExt} and (\ref{AsTotal1su}), the
same Tauberian theorem \cite[Ch. XIII.5, Th.4]{Fe2} applied to the right hand
side of
\begin{equation*}
\lambda ^{-1}\mathbb{E}\left[ 1-e^{-\lambda W_{\rm T}}\right] =\int_{0}^{\infty }%
\mathbb{P}\left[W_{\rm T}>x\right] e^{-\lambda x}dx
\end{equation*}%
yields 
\begin{equation*}
\mathbb{E}\left[ 1-e^{W_{\rm T}\log Q_{n}}\right] \sim K_{\kappa }\cdot
q_{\kappa }\left( n\right) ,\quad n\rightarrow \infty ,
\end{equation*}%
with%
\begin{equation}
K_{\kappa }=\left\{ 
\begin{array}{ccc}
\Gamma \left( 1-\kappa \right) C_{\kappa }\left( \frac{2}{m_{2}}\right)
^{\kappa }, & \text{if } & \kappa <1, \\ 
&  &  \\ 
\frac{2}{m_{2}}C_{1}, & \text{if} & \kappa =1, \\ 
&  &  \\ 
\frac{2}{m_{2}}\int_{0}^{\infty }\mathbb{P}\left( W_{\rm T}>x\right) dx, & \text{%
if} & \kappa >1.%
\end{array}%
\right.   \label{ReprK}
\end{equation}%
\end{proof}

\section{The critical case with a Markovian environment \label{S_CM}}
As compared to the IID case, Markovian environments require extra
conditions on the underlying Markov chain. First we assume that the two-type
critical process $(X_{n},Z_{n})$ evolves in a stationary
Markovian random environment as defined in Section \ref{Sme}. Besides, we
suppose the validity of \eqref{fub} and that for some $\rho >0$ 
\begin{equation}
\mathbb{P}\left[ \tau >x\right] =o\left(x^{-1}(\log x)^{-1-\rho }\right)
,~x\rightarrow \infty .  \label{MM1}
\end{equation}%
This implies that $a:=\mathbb{E}[\tau]<\infty $ and due to Lemma
\ref{pro} conditions $\mathbb{E}[\hat \zeta]=0$ and $%
\mathbb{E}[\zeta]=0$ become equivalent.
Moreover, under condition (\ref{MM1}) the sequence of regeneration times %
\eqref{reg} satisfies 
\begin{equation}
\mathbb{P}\left[ \ \left\vert k^{-1}\tau _{k}-a\right\vert >\varepsilon
\right] =o\left( (\log k)^{-1-\rho }\right) ,\ k\rightarrow \infty ,
\label{reg1}
\end{equation}%
for an arbitrarily small $\varepsilon >0$, cf. \cite{HR67}.

\begin{theorem}\hspace{-1.5mm}.
\label{L_Markov} Assume \eqref{Crit}, \eqref{fub}, \eqref{MM1}, and 
\begin{equation*}
\mathbb{E}[\zeta]=0,\quad \mathrm{Var}[\hat \zeta] \in \left( 0,\infty \right) ,  \label{vars}
\end{equation*}%
\begin{equation*}
\mathbb{E}\left[\hat \mu _{2}\hat  \mu _{1} ^{-2}\left(
1+\max \left( 0,\log \hat \mu _{1}\right) \right) \right] <\infty ,\quad 
\mathbb{E}[ \hat  \mu _{1} ^{-1}]<\infty .
\end{equation*}%
Further, for some positive $\alpha $ let
\begin{align*}
\mathbb{P}\left[ \hat \theta _{1}<1/x\right] & =o\left( (\log
x)^{-3-\alpha }\right) ,\ x\rightarrow \infty ,  
 \\
\mathbb{P}\left[ \hat \theta _{1}>x\right] & =o\left( (\log x)^{-3-\alpha
}\right) ,\ x\rightarrow \infty ,  
\\
\mathbb{P}\left[ \hat \theta _{2}>x\hat \theta _{1}\right] & =o\left(
(\log x)^{-3-\alpha }\right) ,\ x\rightarrow \infty .  
\end{align*}%
Then $\mathbb{P}\left[X_{n}>0\right]=O(n^{-1/2}) $ and there exists a constant $\hat K_0>0$ such that 
\begin{equation*}
\mathbb{P}\left[Z_{n}>0\right]
\sim \frac{\hat K_0}{\log n},\ n\rightarrow \infty .  
\end{equation*}
\end{theorem}

%
%
%
\begin{proof}. 
The statement is derived in two steps: first 
\begin{equation}
\mathbb{P}\left[
\hat Z_{r}>0\right] \sim \frac{\hat K_0}{\log r},\quad r\rightarrow \infty
\label{step1}
\end{equation}%
and then 
\begin{equation}
\mathbb{P}\left[ X_{n}+Z_{n}>0\right]
\sim \frac{\hat K_0}{\log n},\quad n\rightarrow \infty .  \label{step2}
\end{equation}
together with  
\begin{equation}
\mathbb{P}\left[ X_{n}>0\right]=O(n^{-1/2}),\quad n\to\infty. \label{lost}
\end{equation}

Fix $\delta \in \left( 0,1/4\right) $ and write%
\begin{eqnarray*}
\mathbb{P}\left[\hat Z_{r}>0\right] &=&\mathbb{P}\left[\hat  Z_{r}>0;\tau _{r^{\delta }}\leq r^{2\delta },\left\vert \tau _{r}-ra\right\vert
\leq r^{1-\delta }\right] \\
&&+O\left( \mathbb{P}\left[ \tau _{r^{\delta }}>r^{2\delta }\right] \right)
+O\left( \mathbb{P}\left[ \left\vert r^{-1}\tau _{r}-a\right\vert
>r^{-\delta }\right] \right) .
\end{eqnarray*}%
Here the last two terms are treated with the help of $\mathbb{P}\left( \tau
_{r^{\delta }}>r^{2\delta }\right) \leq ar^{-\delta }$ and \eqref{reg1},
while the main term is analyzed by means of ideas from the proof of Theorem %
\ref{L_general}. Letting $\hat Y_{k}$ be the number of type $\hat 2$
daughters produced by $\hat X_{k}$ particles of type $\hat 1$ and
putting $\hat T:=\min \left\{ r: \hat Z_{r}=0\right\} $, we deduce from 
\begin{align*}
\mathbb{P}& \left[ \hat Z_{r}>0;\tau _{r^{\delta }}\leq r^{2\delta
},\left\vert \tau _{r}-ra\right\vert \leq r^{1-\delta }\right]\\
& =\mathbb{E}\left[ 1-\prod\limits_{k=0}^{r-1}Q_{\tau _{r}-\tau
_{k}}^{\hat Y_{k}};\tau _{r^{\delta }}\leq r^{2\delta },\left\vert \tau
_{r}-ra\right\vert \leq r^{1-\delta }\right] ,
\end{align*}%
a lower bound
\begin{align*}
\mathbb{P}& \left[\hat  Z_{r}>0;\tau _{r^{\delta }}\leq r^{2\delta
},\left\vert \tau _{r}-ra\right\vert \leq r^{1-\delta }\right] \\
& \geq \mathbb{E}\left[ 1-\prod\limits_{k=0}^{\hat T-1}Q_{\tau _{r}-\tau
_{k}}^{\hat Y_{k}};\hat T\leq r^{\delta },\tau _{r^{\delta }}\leq
r^{2\delta },\left\vert \tau _{r}-ra\right\vert \leq r^{1-\delta }\right] \\
& \geq \mathbb{E}\left[ 1-\prod\limits_{k=0}^{\hat T-1}f_{k}^{\hat X_{k}}(1,Q_{2ra});\hat T\leq r^{\delta },\tau _{r^{\delta }}\leq r^{2\delta
},\left\vert r^{-1}\tau _{r}-a\right\vert \leq r^{-\delta }\right] \\
& \geq \mathbb{E}\left[ 1-\prod\limits_{k=0}^{\hat T-1}f_{k}^{\hat X_{k}}(1,Q_{2ra})\right] -\mathbb{P}\left(\hat  T>r^{\delta }\right) -\mathbf{%
P}\left( \tau _{r^{\delta }}>r^{2\delta }\right) -\mathbb{P}\left(
\left\vert r^{-1}\tau _{r}-a\right\vert >r^{-\delta }\right) \\
& =\mathbb{E}\left[ 1-\prod\limits_{k=0}^{\hat T-1}f_{k}^{\hat X_{k}}(1,Q_{2ra})\right] +o\left( \frac{1}{\log ^{1+\rho }r}\right) .
\end{align*}%
Hence, applying arguments used to derive (\ref{Up}) in Theorem \ref%
{L_general} one can show that for some $\hat K_0>0$
\begin{equation*}
\limsup_{r\rightarrow \infty }\,\left\{ \log r\cdot \mathbb{P}\left[
\hat Z_{r}>0\right] \right\} \leq \hat K_0 \leq \liminf_{r\rightarrow \infty }\,\left\{ \log r\cdot \mathbb{P}\left[
\hat Z_{r}>0\right] \right\}
\end{equation*}%
proving \eqref{step1}.

To demonstrate that \eqref{step2} follows from \eqref{step1} observe first that due to
\begin{equation}
\mathbb{P}[ \hat X_{r}>0] =O(r^{-1/2}),\quad r\rightarrow \infty,
\label{last} 
\end{equation}
we have 
\[
\mathbb{P}\left[\hat X_{r}+
\hat Z_{r}>0\right] \sim \frac{\hat K_0}{\log r},\quad r\rightarrow \infty.
\]
Setting $N_n:=\max \left\{ k:\tau _{k}\leq n\right\}$ we obtain
\begin{equation}
\mathbb{P}\left[ \hat X_{N_n+1}+\hat Z_{N_n+1}>0\right] \leq 
\mathbb{P}\left[ X_{n}+Z_{n}>0\right] \leq \mathbb{P}\left[ \hat X_{N_n}+\hat Z_{N_n}>0\right],   \label{2ine}
\end{equation}%
and for any $\varepsilon \in \left( 0,1\right) $ we get
\begin{align*}
\mathbb{P}\left[\hat  X_{N_n}+\hat Z_{N_n}>0\right] & =\mathbb{P}%
\left[\hat  X_{N_n}+\hat Z_{N_n}>0;N_n\geq a^{-1}n(1-\varepsilon
)\right] \\
& +O\left( \mathbb{P}\left[ N_n<a^{-1}n(1-\varepsilon )\right] \right) .
\end{align*}%
It follows that
\begin{equation*}
\mathbb{P}\left[\hat  X_{N_n}+\hat Z_{N_n}>0;N_n\geq
a^{-1}n(1-\varepsilon )\right] \leq \mathbb{P}\left[
\hat X_{a^{-1}n(1-\varepsilon )}+\hat Z_{a^{-1}n(1-\varepsilon )}>0\right] .
\end{equation*}%
On the other hand, again by \eqref{reg1} as $n\rightarrow \infty $ 
\begin{equation*}
\mathbb{P}\left[ N_n<a^{-1}n(1-\varepsilon )\right] =\mathbb{P}\left[
\tau _{a^{-1}n(1-\varepsilon )}>n\right] =o\left( (\log n)^{-1-\rho }\right) .
\end{equation*}%
Thus, 
\begin{align*}
\limsup_{n\rightarrow \infty }\,& \left\{ \log n\cdot \mathbb{P}\left[
X_{n}+Z_{n}>0\right] \right\} \\
& \leq \limsup_{n\rightarrow \infty }\,\left\{ \log n\cdot \mathbb{P}\left[
\hat X_{a^{-1}n(1-\varepsilon )}+\hat Z_{a^{-1}n(1-\varepsilon )}>0\right] \right\} \leq \hat K_0.
\end{align*}%
A similar estimate from below follows from 
\begin{align*}
\mathbb{P}& \left[\hat  X_{N_n+1}+\hat Z_{N_n+1}>0\right] \\
& \qquad \geq \mathbb{P}\left[\hat  X_{a^{-1}n(1+\varepsilon )}+\hat Z_{a^{-1}n(1+\varepsilon )}>0;N_n+1\leq a^{-1}n(1+\varepsilon
)\right] \\
& \qquad =\mathbb{P}\left[\hat  X_{a^{-1}n(1+\varepsilon )}+\hat 
Z_{a^{-1}n(1+\varepsilon )}>0\right] +o\left( (\log n)^{-1-\rho }\right).
\end{align*}
Finally,  relation  \eqref{lost}  is derived from  \eqref{last} by the law of large numbers argument.
 
\end{proof}

\section{Subcritical processes with a Markovian environment \label{S_SM}}

Assume now that the two-type subcritical process $\ (X_{n},Z_{n})$ evolves
in a stationary Markovian random environment as defined in Section \ref{Sme}%
. Here, similarly to Section \ref{S_CM} the auxiliary branching process $%
\left(\hat  X_{r}, \hat Z_{r}\right) $ in IID environment with
probability generating functions (\ref{DefF}) and (\ref{Defh}) plays an
important role.

Single type subcritical processes with a Markovian environment were recently studied in \cite{Dy}. According to  \cite{Dy} under the conditions of our next theorem one has, similarly to \eqref{3}, that
\begin{equation*}
\mathbb{P}\left[ X_{n}>0\right] =o\left(
A^{n}\right) \mbox { for some constant
$A\in \left( 0,1\right) $}. 
\end{equation*}%

\begin{theorem}\hspace{-1.5mm}.
\label{L_subMark} Assume that assumption \eqref{Crit} holds, 
\begin{equation}
\mathbb{E}[\zeta]<0,\quad \mathrm{Var}[\hat  \zeta] \in \left( 0,\infty \right) ,  \label{l25}
\end{equation}%
and conditions \eqref{defkap}, \eqref{cond} and either  \eqref{condg0} or %
\eqref{condg} are valid for the corresponding random variables related to
the embedded process $( \hat X_{r}, \hat Z_{r}) $ with the
key constant $\kappa $ replaced by $\hat \kappa>0$. Suppose, in addition,
that 
\begin{equation}
\quad \mathbb{P}\left[ \tau >x\right] =o\left( x^{-1-\min (\hat \kappa,1)}\right), \quad x\rightarrow \infty .  \label{K1}
\end{equation}%
Then, there exists a constant $\hat K\equiv \hat K_{\hat \kappa}>0$, given by \eqref{Kas1} and \eqref{Kas2} below,
such that, see \eqref{qka},
\begin{equation}
\mathbb{P}\left[Z_{n}>0\right] \sim a^{\min(1,\hat \kappa)}\hat  Kq_{\hat \kappa}\left( n\right),\quad n\rightarrow \infty.  \label{SurSM}
\end{equation}
\end{theorem}
\begin{proof}.
Our main arguments here are similar to that used in the
proof of Theorem \ref{L_Markov}. Fix $\varepsilon \in \left( 0,1\right) $
and a sufficiently large $N$ and with $a=\mathbb{E}(\tau)$
write%
\begin{eqnarray*}
\mathbb{P}\left[\hat  Z_{r}>0\right] &=&\mathbb{P}\left[\hat  Z_{r}>0;%
\mathcal{B}\left( r,\varepsilon \right) \right] \\
&&+O\left( \mathbb{P}\left[ \tau _{N\log r}>r^{\hat \kappa}\log
^{3}r\right] +\mathbb{P}\left[ \left\vert \tau _{r}-ra\right\vert
>r\varepsilon \right] \right) .
\end{eqnarray*}%
where%
\begin{equation*}
\mathcal{B}\left( r,\varepsilon \right) :=\left\{ \tau _{N\log r}\leq
r^{\hat \kappa}\log ^{3}r,\left\vert \tau _{r}-ra\right\vert \leq
\varepsilon r\right\}.
\end{equation*}%
Clearly,%
\begin{equation}
\mathbb{P}\left[ \tau _{N\log r}>r^{\hat \kappa}\log ^{3}r\right] \leq 
\frac{Na}{r^{\hat \kappa}\log ^{2}r}=o\left( r^{-\hat \kappa}\right).  \label{Com1}
\end{equation}%
Further, if $\hat \kappa<1$ then, according to \cite{HR67} under condition (\ref{K1}) we have
\begin{equation}
\mathbb{P}\left[\ \left\vert \tau _{r}-ra\right\vert >\varepsilon r\right]
=o\left( r^{-\hat \kappa}\right) .  \label{Com2}
\end{equation}
Thus, 
\begin{eqnarray*}
\mathbb{P}\left[\hat  Z_{r}>0\right] \geq \mathbb{P}\ \left[
\hat Z_{r}>0;\mathcal{B}\left( r,\varepsilon \right) \right] =\mathbb{E}\left[ 1-\prod\limits_{k=0}^{r-1}Q_{\tau _{r}-\tau
_{k}}^{\hat Y_{k}};\mathcal{B}\left( r,\varepsilon \right) \right],
\end{eqnarray*}%
and therefore, denoting by $\mathcal{\bar{B}}\left( r,\varepsilon \right) $ the event complementary to $\mathcal{B}\left( r,\varepsilon \right) $, we get
\begin{eqnarray*}
\mathbb{P}\left[\hat  Z_{r}>0\right] &\geq &\mathbb{E}\left[ 1-\prod\limits_{k=0}^{\hat T-1}Q_{\tau _{r}-\tau
_{k}}^{\hat Y_{k}}; \hat T\leq N\log r;\mathcal{B}\left( r,\varepsilon
\right) \right] \\
&\geq &\mathbb{E}\left[ 1-\prod\limits_{k=0}^{\hat T-1}Q_{ra+2\varepsilon
r}^{\hat Y_{k}};\hat T\leq N\log r;\mathcal{B}\left( r,\varepsilon
\right) \right] \\
&=&\mathbb{E}\left[ 1-e^{\hat W_{\rm T}\log Q_{ra+2\varepsilon r}}\right] -%
\mathbb{P}\left[\hat  T>N\log r\right] -\mathbb{P}\left[ \mathcal{\bar{B}}%
\left( r,\varepsilon \right) \right].
\end{eqnarray*}%
Due to (\ref%
{pro}) and (\ref{l25}) 
\begin{equation}
\mathbb{P}\left[ \hat X_{r}>0\right] =\mathbb{P}\left[ \hat T>r\right]
=o\left( A^{r}\right) \text{ for some }A<1.  \label{nn}
\end{equation}%
It follows that in view of (\ref{pro}), (\ref{Com1}) and (\ref{Com2}) one
can find $N$ such that 
\begin{equation*}
\mathbb{P}\left[\hat  T>N\log r\right] +\mathbb{P}\left[ \mathcal{\bar{B}}%
\left( r,\varepsilon \right) \right] =o\left( r^{-\hat \kappa}\right).
\end{equation*}%
On the other hand, using (\ref{ProbExt}) and 
\begin{equation*}
\mathbb{P}\left[\hat  W_{\rm T}>y\right] \sim \hat  Cy^{-\hat \kappa},\;\hat C\in \left( 0,\infty
\right) ,
\end{equation*}%
one can show, arguing as in Theorem \ref{T_Dsubcrit}, that for $\hat \kappa<1$ 
\begin{align*}
\liminf_{\varepsilon \downarrow 0}& \lim_{r\rightarrow \infty }r^{\hat \kappa}\mathbb{E}\left[ 1-e^{\hat W_{\rm T}\log Q_{ra+2\varepsilon r}}%
\right] \\
& =\liminf_{\varepsilon \downarrow 0}\lim_{r\rightarrow \infty }\frac{%
r^{\hat \kappa}}{-\log Q_{ra+2\varepsilon r}}\int_{0}^{\infty }\mathbb{P}%
\left[\hat  W_{\rm T}>x\right] e^{x\log Q_{ra+2\varepsilon r}}dx \\
& =\hat C\Gamma \left( 1-\hat \kappa\right)
\liminf_{\varepsilon \downarrow 0}\left( \frac{2}{m_{2}a\left( 1+2\varepsilon \right) }\right)
^{\hat \kappa}
\end{align*}%
giving 
\begin{equation*}
\liminf_{r\rightarrow \infty }r^{\hat \kappa}\mathbb{P}\left[
\hat Z_{r}>0\right] \geq \hat K
\end{equation*}%
with
\begin{equation}
\hat K=\hat C\Gamma \left( 1-\hat \kappa\right) \left( \frac{2}{m_{2}a}\right) ^{\hat \kappa} \mbox{ for }\hat \kappa<1.
\label{Kas1}
\end{equation}%
A similar upper bound in view of (\ref{EstEqv}) and (\ref{nn}) yields 
\begin{equation*}
\lim_{r\rightarrow \infty }r^{\hat \kappa}\mathbb{P}\left[\hat  X_{r}+\hat Z_{r}>0\right] =\hat K.
\end{equation*}

If $\hat \kappa\geq 1$ then condition (\ref{K1}) entails%
\begin{equation*}
\mathbb{P}\left[ \ \left\vert \tau _{r}-ra\right\vert >\varepsilon r\right]
=o\left( r^{-1}\right),
\end{equation*}%
and, as before, this implies 
\begin{equation*}
\lim_{r\rightarrow \infty }\left( q_{\hat \kappa}(r)\right) ^{-1}\mathbf{%
P}\left[\hat  X_{r}+\hat Z_{r}>0\right] =\hat K,
\end{equation*}%
where 
\begin{equation}
\hat K=\frac{2}{m_{2}a}\cdot \left\{ 
\begin{array}{lll}
\hat C, & \text{if} & \hat \kappa=1, \\ 
\int_{0}^{\infty }\mathbb{P}\big[\hat  W_{\rm T}>x\big] dx, & \text{if} & 
\hat \kappa>1.%
\end{array}%
\right.
\label{Kas2}
\end{equation}

We proceed by recalling \eqref{2ine}. For any $\varepsilon \in \left(
0,1\right) $ 
\begin{align*}
\mathbb{P}\left[\hat  X_{N_n}+\hat Z_{N_n}>0\right] & =\mathbb{P}%
\left[\hat  X_{N_n}+\hat  Z_{N_n}>0;N_n\geq a^{-1}n(1-\varepsilon
)\right] \\
& +O\left( \mathbb{P}\left[N_n<a^{-1}n(1-\varepsilon )\right] \right) ,
\end{align*}%
and as $n\rightarrow \infty $ 
\begin{align*}
\mathbb{P}& \left[\hat  X_{N_n}+\hat Z_{N_n}>0;N_n\geq a^{-1}n(1-\varepsilon )\right] \\
& \quad \leq \mathbb{P}\left[\hat  X_{a^{-1}n\left( 1-\varepsilon \right) }+\hat Z_{a^{-1}n\left( 1-\varepsilon \right) }>0\right] \sim q_{\hat \kappa}\left( a^{-1}n\left( 1-\varepsilon \right) \right) \hat K.
\end{align*}%
It follows from \cite{HR67} and our conditions that 
\begin{equation*}
\mathbb{P}\left[ N_n<a^{-1}n(1-\varepsilon )\right] =\mathbb{P}\left[
S_{a^{-1}n\left( 1-\varepsilon \right) }>n\right] =o\left( n^{-\min (\hat \kappa,1)}\right) ,\ n\rightarrow \infty .
\end{equation*}%
Thus,%
\begin{align*}
\limsup_{n\rightarrow \infty }& \left( q_{\hat \kappa}(n)\right) ^{-1}%
\mathbb{P}\left[ X_{n}+Z_{n}>0\right] \\
& \leq \lim_{\varepsilon \downarrow 0}\limsup_{n\rightarrow \infty }\left(
q_{\hat \kappa}(n)\right) ^{-1}\mathbb{P}\left[ \hat X_{N_n}+\hat Z_{N_n}>0\right],
\end{align*}%
so that 
\begin{equation*}
\limsup_{n\rightarrow \infty }\left( q_{\hat \kappa}(n)\right) ^{-1}%
\mathbb{P}\left[ X_{n}+Z_{n}>0\right] \leq a^{\min(1,\hat \kappa)}\hat K.
\end{equation*}%
The corresponding lower bound is obtained similarly.
\end{proof}

\section*{Acknowledgments}

ED was supported in part by the Program of RAS \textquotedblleft Theoretical Problems of Contemporary Mathematics\textquotedblright\ and by INTAS grant No.
03-51-5018. SS was supported by Swedish Research Council grant No.
621-2010-5623. VV was supported in part by the Program of RAS
\textquotedblleft Theoretical Problems of Contemporary Mathematics\textquotedblright\
and by a Wenner-Gren Foundation Visiting Researcher Scholarship.

%
%
%
%
%

\end{document}